\documentclass[12pt,reqno]{amsart}
\usepackage{amscd,amssymb,amsmath,multicol}

\newtheorem{thm}[equation]{Theorem}
\numberwithin{equation}{section}
\newtheorem{cor}[equation]{Corollary}

\newtheorem{lem}[equation]{Lemma}

\newtheorem{defin}[equation]{Definition}

\newtheorem{prop}[equation]{Proposition}

\begin{document}
\raggedbottom \voffset=-.7truein \hoffset=0truein \vsize=8truein
\hsize=6truein \textheight=8truein \textwidth=6truein
\baselineskip=18truept
\def\vareps{\varepsilon}
\def\mapright#1{\ \smash{\mathop{\longrightarrow}\limits^{#1}}\ }
\def\mapleft#1{\smash{\mathop{\longleftarrow}\limits^{#1}}}
\def\mapup#1{\Big\uparrow\rlap{$\vcenter {\hbox {$#1$}}$}}
\def\mapdown#1{\Big\downarrow\rlap{$\vcenter {\hbox {$\ssize{#1}$}}$}}
\def\on{\operatorname}
\def\spa{\on{span}}
\def\a{\alpha}
\def\bz{{\Bbb Z}}
\def\gd{\on{gd}}
\def\imm{\on{imm}}
\def\sq{\on{Sq}}
\def\ssp{\on{stablespan}}
\def\eps{\epsilon}
\def\br{{\Bbb R}}
\def\bc{{\Bbb C}}
\def\bh{{\Bbb H}}
\def\tfrac{\textstyle\frac}
\def\w{\wedge}
\def\b{\beta}
\def\A{{\cal A}}
\def\P{{\cal P}}
\def\zt{{\Bbb Z}_2}
\def\bq{{\Bbb Q}}
\def\ker{\on{ker}}
\def\coker{\on{coker}}
\def\u{{\cal U}}
\def\e{{\cal E}}
\def\exp{\on{exp}}
\def\wbar{{\overline w}}
\def\xbar{{\overline x}}
\def\ybar{{\overline y}}
\def\zbar{{\overline z}}
\def\ebar{{\overline e}}
\def\nbar{{\overline n}}
\def\mbar{{\overline m}}
\def\ubar{{\overline u}}
\def\et{{\widetilde E}}
\def\pt{{\widetilde P}}
\def\rt{{\widetilde R}}
\def\vt{{\widetilde\nu}}
\def\ni{\noindent}
\def\coef{\on{coef}}
\def\den{\on{den}}
\def\gd{{\on{gd}}}
\def\N{{\Bbb N}}
\def\Z{{\Bbb Z}}
\def\Q{{\Bbb Q}}
\def\R{{\Bbb R}}
\def\C{{\Bbb C}}
\def\Ah{\widehat{A}}
\def\Bh{\widehat{B}}
\def\Ch{\widehat{C}}
\def\Bin{\on{Bin}}
\def\xmin{x_{\text{min}}}
\def\xmax{x_{\text{max}}}
\title[Coefficients in powers of log series]
{Coefficients in powers of the log series}

\author{Donald M. Davis}
\address{Department of Mathematics, Lehigh University\\Bethlehem, PA 18015, USA}
\email{dmd1@lehigh.edu}
\date{January 16, 2010}


\maketitle
\begin{abstract} We determine the $p$-exponent in many of the coefficients of $\ell(x)^t$,
where $\ell(x)$ is the power series for $\log(1+x)/x$ and $t$ is any integer. In our proof, we introduce a variant of
multinomial coefficients. We also characterize the power series $x/\log(1+x)$ by certain zero coefficients in its powers.
 \end{abstract}

\section{Main divisibility theorem}\label{intro}
The divisibility by primes of the coefficients in the integer powers $\ell(x)^t$ of the power series for $\log(1+x)/x$,
given by
$$\ell(x):=\sum_{i=0}^\infty(-1)^i\frac{x^i}{i+1},$$
has been applied in several ways in algebraic topology.  See, for example, \cite{AT} and \cite{SS}.
Our main divisibility result, \ref{mainthm}, says that, in an appropriate range, this divisibility is the same as that
of the coefficients of $(1\pm \frac{x^{p-1}}p)^t$. Here $p$ is any prime and $t$ is any integer. We denote by $\nu_p(-)$ the exponent
of $p$ in an integer, and by $[x^n]f(x)$ the coefficient of $x^n$ in a power series $f(x)$.
\begin{thm} \label{mainthm} If $t$ is any integer and $m\le p^{\nu_p(t)}$, then $$\nu_p\left([x^{(p-1)m}]\ell(x)^t\right)=\nu_p(t)-\nu_p(m)-m.$$
\end{thm}
Thus, for example, if $\nu_3(t)=2$, then, for $m=1,\ldots,9$, the exponent of 3 in $[x^{2m}]\ell(x)^t$ is, respectively,
$1$, $0$, $-2$, $-2$, $-3$, $-5$, $-5$, $-6$, and $-9$, which is the same as in $(1\pm \frac{x^2}3)^t$.
In Section \ref{sec2}, we will discuss what we can say about $\nu_p([x^n]\ell(x)^t)$ when $n$ is not divisible by $(p-1)$ and $n<(p-1)p^{\nu_p(t)}$ .

The motivation for Theorem \ref{mainthm} was provided by ongoing thesis work of Karen McCready at Lehigh University, which
seeks to apply the result when $p=2$ to make more explicit some nonimmersion results for complex projective
spaces described in \cite{SS}. Proving Theorem \ref{mainthm} led the author to discover an interesting modification
of multinomial coefficients.

\begin{defin} For an ordered $r$-tuple of nonnegative integers $(i_1,\ldots,i_r)$, we define
$$c(i_1,\ldots,i_r):=\frac{(\sum i_jj)(\sum i_j -1)!}{i_1!\cdots i_r!}.$$
\end{defin}
Note that $c(i_1,\ldots,i_r)$ equals $(\sum i_jj)/\sum i_j$ times a multinomial coefficient.
Surprisingly, these numbers satisfy the same recursive formula as  multinomial coefficients.
\begin{defin} For positive integers $k\le r$, let $E_k$ denote the
ordered $r$-tuple whose only nonzero entry is a 1 in position $k$.\end{defin}
\begin{prop} \label{recurs} If $I=(i_1,\ldots,i_r)$ is an ordered $r$-tuple of nonnegative integers, then
\begin{equation}\label{receq}c(I)=\sum_{i_k>0}c(I-E_k).\end{equation}
\end{prop}
If we think of a multinomial coefficient $\binom{\sum i_j}{i_1,\cdots,i_r}:=(i_1+\cdots+i_r)!/((i_1)!\cdots(i_r)!)$ as being determined
by the unordered $r$-tuple $(i_1,\ldots,i_r)$ of nonnegative integers, then it satisfies the recursive
formula analogous to that of  (\ref{receq}). For a multinomial coefficient, entries which are 0
can be omitted, but that is not the case for $c(i_1,\ldots,i_r)$.

\begin{proof}[Proof of Proposition \ref{recurs}] The right hand side of (\ref{receq}) equals
\begin{eqnarray*}&&\sum_k i_k\frac{(\sum i_j -2)!}{(i_1)!\cdots(i_r)!}\left(\sum_j i_jj -k\right)\\
&=&\frac{(\sum i_j -2)!}{(i_1)!\cdots(i_r)!}\left(\left(\sum i_k\right)\left(\sum i_jj\right)-\sum i_kk\right)\\
&=&\frac{(\sum i_j -2)!}{(i_1)!\cdots(i_r)!}\left(\sum i_jj\right)\left(\sum i_j-1\right),
\end{eqnarray*}
which equals the left hand side of (\ref{receq}).
\end{proof}
\begin{cor} If $\sum i_j>0$, then $c(i_1,\ldots,i_r)$ is a positive integer.\label{cor1}\end{cor}
\begin{proof} Use (\ref{receq}) recursively to express $c(i_1,\ldots,i_r)$ as a sum of various $c(E_k)=k$.
\end{proof}
\begin{cor}\label{corr2} For any ordered $r$-tuple $(i_1,\ldots,i_r)$ of nonnegative integers and any prime $p$,
\begin{equation}\nu_p\left(\sum i_j\right)\le\nu_p\left(\sum i_jj\right)+\nu_p\binom{\sum i_j}{i_1,\cdots,i_r}.\label{cor2}\end{equation}
\end{cor}
\begin{proof} Multiply numerator and denominator of the definition of $c(i_1,\ldots,i_r)$ by $\sum i_j$
and apply Corollary \ref{cor1}.
\end{proof}

The proof of Theorem \ref{mainthm}  utilizes Corollary \ref{corr2} and also the following lemma.
\begin{lem} \label{lem} If $t$ is any integer and $\sum i_j\le p^{\nu_p(t)}$, then
\begin{equation}\label{nueq}\nu_p\binom{t}{t-\sum i_j,i_1,\ldots,i_r}=\nu_p(t)+\nu_p\binom{\sum i_j}{i_1,\ldots,i_r}-\nu_p\left(\sum i_j\right).\end{equation}
\end{lem}
\begin{proof} For any integer $t$, the multinomial coefficient on the left hand side of (\ref{nueq}) equals
$t(t-1)\cdots(t+1-\sum i_j)/\prod i_j!$, and so
the left hand side of (\ref{nueq}) equals $\nu_p(t(t-1)\cdots(t+1-\sum i_j))-\sum\nu_p(i_j!)$.
Since $\nu_p(t-s)=\nu_p(s)$ provided $0<s<p^{\nu_p(t)}$, this becomes $\nu_p(t)+\nu_p((\sum i_j -1)!)-\sum\nu_p(i_j!)$,
and this equals the right hand side of (\ref{nueq}).\end{proof}

\begin{proof}[Proof of Theorem \ref{mainthm}] By the multinomial theorem,
$$[x^{(p-1)m}]\ell(x)^t=(-1)^{(p-1)m}\sum_IT_I,$$ where \begin{equation}T_I=\binom{t}{t-\sum i_j,i_1,\ldots,i_r}\frac1{\prod(j+1)^{i_j}},\label{Teq}\end{equation}
with the sum taken over all $I=(i_1,\ldots,i_r)$ satisfying $\sum i_jj=(p-1)m$.
Using Lemma \ref{lem}, we have
$$\nu_p(T_I)=\nu_p(t)+\nu_p\binom{\sum i_j}{i_1,\ldots,i_r}-\nu_p\left(\sum i_j\right)-\sum i_j\nu_p(j+1).$$
If $I=mE_{p-1}$, then $\nu_p(T_I)=\nu_p(t)+0-\nu_p(m)-m$. The theorem will follow once we show that
all other $I$ with $\sum i_jj=(p-1)m$ satisfy $\nu_p(T_I)>\nu_p(t)-\nu_p(m)-m$. Such $I$ must have
$i_j>0$ for some $j\ne p-1$.  This is relevant because $\frac1{p-1}j\ge \nu_p(j+1)$ with equality if and only if
$j=p-1$. For  $I$ such as we are considering, we have
\begin{eqnarray}&&\nu_p(T_I)-(\nu_p(t)-\nu_p(m)-m)\label{eqs}\\
&=&\nu_p\binom{\sum i_j}{i_1,\ldots,i_r}-\nu_p(\sum i_j)-\sum i_j\nu_p(j+1)+\nu_p(\sum i_jj)+\tfrac1{p-1}\sum i_jj\nonumber\\
&\ge&\sum i_j(\tfrac1{p-1}j-\nu_p(j+1))\nonumber\\
&>&0.\nonumber\end{eqnarray}
We have used (\ref{cor2}) in the middle step.
\end{proof}

\section{Zero coefficients}\label{zsec}
While studying coefficients related to Theorem \ref{mainthm}, we noticed the following result about occurrences
of coefficients of powers of the reciprocal log series which equal 0.
\begin{thm}\label{cor1p} If $m$ is odd and $m>1$, then $[x^m]\bigl(\frac x{\log(1+x)}\bigr)^m=0$, while if $m$ is even and $m>0$, then $[x^{m+1}]\bigl(\frac x{\log(1+x)}\bigr)^m=0$.\end{thm}

Moreover, this property  characterizes the reciprocal log series.
\begin{cor}\label{cor2p} A power series $f(x)=1+\sum\limits_{i\ge1}c_ix^i$ with $c_1\ne0$ has  $[x^m](f(x)^m)=0$
for all odd $m>1$, and $[x^{m+1}](f(x)^m)=0$ for all even $m>0$ if and only if $f(x)=\frac{2c_1x}{\log(1+2c_1x)}$.\end{cor}
\begin{proof} By Theorem \ref{cor1p}, the reciprocal log series satisfies the stated property. Now assume that $f$ satisfies this property and let $n$ be a positive integer and $\eps=0$ or 1. Since $$[x^{2n+1}]f(x)^{2n+\eps}=(2n+\eps)(2n+\eps-1)c_1c_{2n}+(2n+\eps)c_{2n+1}+P,$$
where $P$ is a polynomial in $c_1,\ldots,c_{2n-1}$, we see that
$c_{2n}$ and $c_{2n+1}$ can be determined from the $c_i$ with $i<2n$.
\end{proof}

Our proof of Theorem \ref{cor1p} is an extension of arguments of \cite{AT} and \cite{Dlog}.
It benefited from ideas of Francis Clarke. It can be derived from results in \cite[ch.6]{GKP}, but
we have not seen it explicitly stated anywhere.

\begin{proof}[Proof of Theorem \ref{cor1p}]
 Let $m>1$ and
$$\left(\frac x{\log(1+x)}\right)^m=\sum_{i\ge0}a_i x^i.$$
Letting $x=e^y-1$, we obtain
\begin{equation}\label{eq}\left(\frac{e^y-1}y\right)^m=\sum_{i\ge0}a_i(e^y-1)^i.\end{equation}
Let $j$ be a positive integer, and multiply both sides of (\ref{eq}) by $y^me^y/(e^y-1)^{j+1}$, obtaining
\begin{eqnarray}\label{eq1}(e^y-1)^{m-j-1}e^y&=&y^m\sum_{i\ge0}a_i(e^y-1)^{i-j-1}e^y\\
&=&y^m\left(a_j\frac{e^y}{e^y-1}+\sum_{i\ne j}\tfrac{a_i}{i-j}\tfrac d{dy}(e^y-1)^{i-j}\right).\nonumber\end{eqnarray}
  Since the derivative of a Laurent series has no $y^{-1}$-term, we conclude
that the coefficient of $y^{m-1}$ on the RHS of (\ref{eq1}) is $a_j[y^{-1}](1+\frac 1y\frac y{e^y-1})=a_j$.

The Bernoulli numbers $B_n$ are defined by $\frac y{e^y-1}=\sum \frac{B_n}{n!}y^n$.
Since $\frac y{e^y-1}+\frac12y$ is an even function of $y$, we have the well-known result that
$B_n=0$ if $n$ is odd and $n>1$.

Let
$$j=\begin{cases}m&m\text{ odd}\\
m+1&m\text{ even.}\end{cases}$$
For this $j$, the LHS of (\ref{eq1}) equals
$$\begin{cases}1+\sum\frac{B_i}{i!}y^{i-1}&m\text{ odd}\\
-\frac d{dy}(e^y-1)^{-1}=-\sum\frac{(i-1)B_i}{i!}y^{i-2}&m\text{ even,}\end{cases}$$
and comparison of coefficient of $y^{m-1}$ in (\ref{eq1}) implies
$$\begin{cases}a_m=\frac{B_m}{m!}=0&m\text{ odd}\\
a_{m+1}=-\frac{mB_{m+1}}{(m+1)!}=0&m\text{ even,}\end{cases}$$
yielding the theorem.\end{proof}

\section{Other coefficients}\label{sec2}
In this section, a sequel to Theorem \ref{mainthm}, we describe what can be easily said about
$\nu_p([x^{(p-1)m+\Delta}]\ell(x)^t)$ when $0<\Delta<p-1$ and $m< p^{\nu_p(t)}$.
This is not relevant in the motivating case, $p=2$. Our first result says that these exponents are at least as large
as those of
$[x^{(p-1)m}]\ell(x)^t$.
Here $t$ continues to denote any integer, positive or negative.
\begin{prop} If $0<\Delta<p-1$ and $m< p^{\nu_p(t)}$, then
$$\nu_p\left([x^{(p-1)m+\Delta}]\ell(x)^t\right)\ge\nu_p(t)-\nu_p(m)-m.$$\label{delprop}
\end{prop}
\begin{proof} We consider terms $T_I$ as in (\ref{Teq}) with $\sum i_jj=(p-1)m+\Delta$.
Similarly to (\ref{eqs}), we obtain
\begin{eqnarray} &&\nu_p(T_I)-(\nu_p(t)-\nu_p(m)-m)\nonumber\\
&=&\nu_p\binom{\sum i_j}{i_1,\ldots,i_r}-\nu_p\left(\sum i_j\right)-\sum i_j\nu_p(j+1)\nonumber\\
&&+\nu_p(m)+m\label{A}.\end{eqnarray}

For $I=(i_1,\ldots,i_r)$, let
\begin{eqnarray*}\vt_p(I)&:=&\nu_p\binom{\sum i_j}{i_1,\ldots,i_r}-\nu_p\left(\sum i_j\right)\\
&=&\nu_p\left(\frac 1{i_j}\binom{\sum i_j-1}{i_1,\ldots,i_j-1,\ldots,i_r}\right),\end{eqnarray*}
for any  $j$. Thus
\begin{equation}\label{vt}\vt_p(I)\ge-\min\limits_j\nu_p(i_j).\end{equation}

Ignoring the term $\nu_p(m)$, the expression (\ref{A}) is
\begin{equation}\label{C}\ge\vt_p(I)+\sum i_j(\tfrac1{p-1}j-\nu_p(j+1))-\tfrac{\Delta}{p-1}.\end{equation}
Note that
$$\sum i_j(\tfrac1{p-1}j-\nu_p(j+1))-\tfrac{\Delta}{p-1}=m-\sum i_j\nu_p(j+1)$$
is an integer and is greater than $-1$, and hence is $\ge0$.

By (\ref{vt}), if $\vt_p(I)=-e$ with $e\ge0$, then all $i_j$ are divisible by $p^e$. Thus $\sum i_j(\tfrac1{p-1}j-\nu_p(j+1))$
is positive and divisible by $p^e$. Hence it is $\ge p^e$. Therefore, (\ref{C}) is $\ge-e+p^e-1\ge0$.
We obtain the desired conclusion, that, for each $I$, (\ref{C}), and hence (\ref{A}), is $\ge0$.
\end{proof}

Finally, we address the question of when does equality occur in Proposition \ref{delprop}.
We give a three-part result, but by the third it becomes clear that obtaining additional results is probably more
trouble than it is worth.
\begin{prop} In Proposition \ref{delprop},
\begin{itemize}
\item[a.] the inequality is strict ($\ne$) if $m\equiv0\ (p)$;
\item[b.] equality holds if $\Delta=1$ and $m\not\equiv0,1 \ (p)$;
\item[c.] if $\Delta=2$ and $m\not\equiv0,2\ (p)$, then equality holds
if and only if $3m\not\equiv 5\ (p)$.
\end{itemize}
\end{prop}
\begin{proof} We begin as in the proof of \ref{delprop}, and note that, using (\ref{cor2}), (\ref{A}) is
\begin{equation}\ge\nu_p(m)-\tfrac{\Delta}{p-1}+\sum i_j(\tfrac1{p-1}j-\nu_p(j+1))-\nu_p((p-1)m+\Delta)\label{B}.
\end{equation}

(a) If $\nu_p(m)>0$, then $\nu_p((p-1)m+\Delta)=0$ and so (\ref{B}) is greater than 0.

In (b) and (c), we exclude consideration of the case where $m\equiv\Delta\ (p)$ because then
$\nu_p((p-1)m+\Delta)>0$ causes complications.

(b) If $\Delta=1$ and $m\not\equiv0,1\ (p)$, then for $I=E_1+mE_{p-1}$, (\ref{A}) equals
$$\nu_p(m+1)-\nu_p(m+1)-m+\nu_p(m)+m=0,$$
while for other $I$, (\ref{B}) is
$$0-\tfrac1{p-1}+\sum i_j(\tfrac1{p-1}j-\nu_p(j+1))>0.$$

(c) Assume $\Delta=2$ and $m\not\equiv0,2\ (p)$. Then
\begin{eqnarray}&&T_{2E_1+mE_{p-1}}+T_{E_2+mE_{p-1}}\nonumber\\
&=&\frac{t(t-1)\cdots(t-m-1)}{2!m!}\frac1{4p^m}+\frac{t(t-1)\cdots(t-m)}{m!}\frac1{3p^m}\nonumber\\
&=&(-1)^m\tfrac t{p^m}\bigl(\tfrac18(-m-1+A)+\tfrac13(1+B)\bigr)\nonumber\\
&=&(-1)^m\tfrac t{24p^m}(-3m+5+(3A+8B)).\label{D}\end{eqnarray}
Here $A$ and $B$ are rational numbers which are divisible by $p$. This is true because $\nu_p(t)>\nu_p(i)$
for all $i\le m$.  Since $p>3$, (\ref{D}) has $p$-exponent $\ge\nu_p(t)-m$
with equality if and only if $3m-5\not\equiv0\ (p)$. Using (\ref{B}), the other terms $T_I$ satisfy
$$\nu_p(T_I)-(\nu_p(t)-m)\ge\sum i_j(\tfrac1{p-1}j-\nu_p(j+1))-\tfrac2{p-1}>0.$$
\end{proof}

\def\line{\rule{.6in}{.6pt}}

\end{document}